\documentclass[a4paper,reqno]{amsart}

\usepackage{amsfonts}
\usepackage{mathrsfs}
\usepackage{graphicx}
\usepackage[normalem]{ulem}
\usepackage{url}

\usepackage{multirow}
\usepackage{rotating}
\usepackage{booktabs}

\usepackage{amsthm,amsfonts,amssymb,amsmath,esint,amscd,latexsym,amsxtra}
\usepackage[colorlinks=true, urlcolor=blue,bookmarks=true,bookmarksopen=true,
citecolor=blue]{hyperref}
\usepackage{amscd}
\usepackage{enumerate,bbm}

\usepackage{fullpage}
\usepackage[all]{xy}
\usepackage[OT2,T1]{fontenc}
\usepackage{xspace}

\newcommand{\BC}{{\mathbb {C}}}

 \newcommand{\BN}{{\mathbb {N}}}
 
\newcommand{\BQ}{{\mathbb {Q}}}

 \newcommand{\BZ}{{\mathbb {Z}}}

\newcommand{\CE}{{\mathcal {E}}} 
 \newcommand{\CH}{{\mathcal {H}}}
 \newcommand{\CJ}{{\mathcal {J}}}

\newcommand{\CO}{{\mathcal {O}}} 
 
\newcommand{\CS}{{\mathcal {S}}} 
 \newcommand{\CV}{{\mathcal {V}}}
\newcommand{\CW}{{\mathcal {W}}}

\newcommand{\fs}{{\mathfrak{s}}}

\newcommand{\fM}{{\mathfrak{M}}}

 \newcommand{\fX}{{\mathfrak{X}}}

\newcommand{\Aut}{{\mathrm{Aut}}}

\newcommand{\BBC}{{\mathrm{BC}}}

\newcommand{\diag}{{\mathrm{diag}}}
 
\newcommand{\End}{{\mathrm{End}}} 

\newcommand{\Frac}{{\mathrm{Frac}}}

 \newcommand{\GL}{{\mathrm{GL}}}

\newcommand{\Hom}{{\mathrm{Hom}}}

\newcommand{\id}{{\mathrm{id}}}
\newcommand{\Ind}{{\mathrm{Ind}}}

\renewcommand{\Re}{{\mathrm{Re}}}

\newcommand{\Res}{{\mathrm{Res}}}

\newcommand{\Spec}{{\mathrm{Spec\hspace{2pt}}}}

\newcommand{\RTr}{{\mathrm{Tr}}}
\newcommand{\val}{{\mathrm{val}}}

\newcommand{\wt}[1]{{\widetilde {#1}}}

\newcommand{\sk}{\medskip}
\newcommand{\lra}{\longrightarrow}

\newcommand{\bs}{\backslash}

\newcommand{\s}{\sk\noindent}

\makeatletter\def\varW@#1#2{%
\vtop{\m@th\ialign{##\cr
\hfil$#1 \mathrm{colim} $\hfil\cr
\noalign{\nointerlineskip\kern1.5\ex@}#2\cr
\noalign{\nointerlineskip\kern-\ex@}\cr}}
}
\makeatother
\makeatletter
\def\colim{%
\mathop{\mathpalette\varW@{}}\nmlimits@
}\makeatother

\theoremstyle{plain}
\newtheorem{thm}{Theorem}[section] \newtheorem{cor}[thm]{Corollary}

\newtheorem{lem}[thm]{Lemma}  \newtheorem{prop}[thm]{Proposition}
\newtheorem {conj}[thm]{Conjecture}

\theoremstyle{remark} \newtheorem{remark}[thm]{Remark}
\theoremstyle{definition} 
\theoremstyle{definition}  
\newtheorem{defn}[thm]{Definition}\newtheorem{defn+lem}[thm]{Definition and Lemma}

\numberwithin{equation}{section}

\newcommand*{\sheafhom}{\mathrm{H}\kern -.5pt om}
\begin{document}
\title{Spherical Characters in Families: the unitary Gan-Gross-Prasad case}

\begin{abstract}
We consider the variation of spherical characters in families. We formulate
 conjectures for the rationality and meromorphic property of spherical characters. As an example, we  establish these conjectures in the unitary Gan-Gross-Prasad case. 
\end{abstract}

\author{Li Cai}
\address{Academy for Multidisciplinary Studies\\
Beijing National Center for Applied Mathematics\\
Capital Normal University\\
Beijing, 100048, People's Republic of China}
\email{caili@cnu.edu.cn}

\author{Yangyu Fan}
\address{Academy for Multidisciplinary Studies\\
Beijing National Center for Applied Mathematics\\
Capital Normal University\\
Beijing, 100048, People's Republic of China}
\email{b452@cnu.edu.cn}

\maketitle

\tableofcontents

\section{Spherical characters in families}
Let $F$ be a $p$-adic field and $H\subset G$ be a pair of $F$-reductive groups such that  the geometric quotient $Y:=H\bs G$ is {\em spherical}. 

Let $\sigma$ be a complex irreducible unitary $G(F)$-representation  appearing in the Plancherel decomposition 
of $L^2(Y(F))$ with a $G(F)$-invariant pairing on $\sigma$ and its contragradient $\sigma^\vee$. Assume the local
conjecture of Sakellaridis-Venkatesh \cite{SV}. There is 
a {\em canonical local period} 
\[Z_\sigma(\cdot,\cdot):\ \sigma\times\sigma^\vee\to\BC\]
in the bilinear space $\Hom_{H(F)^2}(\sigma \times \sigma^\vee, \BC)$. Globally, under certain conditions (for example, 
$(G,H)$ is a Gelfand pair), it is conjectured that for any cuspidal automorphic representation over $G$,
(the square of) its period integrals over $H$ admit an Eulerian decomposition into the product of the above canonical local 
periods and  {\em certain $L$-value}. 

In some circumstances, especially in the relative trace formula framework, it is more convenient to consider
the {\em spherical character} attached to the caonical local period $Z_\sigma$
$$J_\sigma: \CS(G(F),\BC)\to \BC,\quad f\in \CH(K,\BC)\mapsto \sum_i Z_\sigma(\sigma(f)\varphi_i,\varphi^i).$$
Here,  
\begin{itemize}
\item for a coefficient field $E$, $\CS(G(F),E)$ is the Hecke algebra of $E$-valued compactly supported functions on 
$G(F)$ and for any open compact subgroup $K\subset G(F)$, $\CH(K,E)\subset \CS(G(F),E)$ is the subspace of 
bi-$K$-invariant functions.
\item $\{\varphi_i\}$ and $\{\varphi^i\}$ are bases of $\sigma^K$ and $\sigma^{\vee,K}$ respectively such that  
	$(\varphi_i,\varphi^j)=\delta_{ij}$. 
\item $\sigma(f)\varphi = \int_{G(F)} f(g)\sigma(g)\varphi dg$ is the induced action of $\CS(G(F),\BC)$ on $\sigma$.
\end{itemize}
Note that $J_\sigma$ is actually independent of the choice of $(\cdot,\cdot)$. The above Eulerian decomposition
of period integrals into  canonical local periods and $L$-value is equivalent to the Eulerian decomposition of 
the attached  Bessel distributions into spherical characters and the same $L$-value (See \cite[Lemma 1.7]{WeZ14}).

In \cite{CF21}, motivatived by the construction of $p$-adic $L$-functions from the Eulerian decomposition
of periods integrals, we consider the rationality and meromorphic properties of canonical local periods for
families of representations.

Let $E$ be a fixed coefficient  field embeddable into $\BC$, $R$ be a finitely reduced Noetherian $E$-algebra and $\Sigma\subset X=\Spec(R)$ be a fixed Zariski dense subset of closed points. Let $\pi$ be a finitely generated smooth admissible torsion-free  $R[G(F)]$-module, i.e. $\pi$ is a torsion-free $R$-module equipped with a $R$-linear action of $G(F)$ such that 
\begin{itemize}
    	\item (smooth)   any $\varphi \in \pi$ is fixed by some open compact subgroup of $G(F)$;
	\item (admissible)  the $R$-submodule
	$\pi^K\subset \pi$ of $K$-fixed elements is  finitely generated for any compact open subgroup $K \subset G(F)$;
	\item (finitely generated)  $\pi$ is finitely generated as a $R[G(F)]$-module.
\end{itemize}
For any closed point $x\in \Sigma$ with residue field $k(x)$, denote by $\CE(\pi|_x)$ the set 
	    of field embedding $\tau:\ k(x)\to\BC$ such that the base change $(\pi|_x)_\tau=\pi|_x\otimes_{k(x),\tau}\BC$ for the specialization of $\pi$ at $x$ is irreducible and appears in $L^2(Y(F))$.
\begin{conj}\cite[Conjecture 1.1]{CF21}\label{conj-2}
Assume  $\CE(\pi|_x)$ is non-empty for any $x\in \Sigma$ and moreover
	\begin{itemize}
    \item there exists a finitely generated smooth admissible torsion-free  $R[G(F)]$-module  $\wt{\pi}$ together
	    with a $G(F)$-invariant $R$-bilinear pairing
$( \cdot,\cdot ): \pi \times \tilde{\pi} \lra R$  such that for any $x\in\Sigma$,  $(\cdot,\cdot)$ 
induces  a non-degenerate $G(F)$-invariant pairing $( \cdot,\cdot )|_x:\ \pi|_x\times\tilde{\pi}|_x\to k(x).$
	\end{itemize}
	Then 
	\begin{enumerate}
		\item (Rationality) for any $x\in\Sigma$, there exists a unique bi-$H(F)$-invariant pairing 
		$$Z_{\pi|_x}:\ \pi|_x\times \tilde{\pi}|_x\to k(x)$$
	such that for any $\tau\in\CE(\pi|_x)$,	 the following diagram commutes
			\[\xymatrix{
			&\pi|_x \times \tilde{\pi}|_x \ar[d]^-{\tau} \ar[r]^-{Z_{\pi|_x}} & k(x) \ar[d]^-{\tau} \\
			&(\pi|_x)_\tau \times (\tilde{\pi}|_x)_\tau \ar[r]^-{Z_{(\pi|_x)_\tau}} & \BC
			.}\]
		 Here $Z_{(\pi|_x)_\tau}$ is defined with respect to the linear extension of $(\cdot,\cdot)|_x$.
		
		\item (Meromorphy) upon shrinking $\Spec(R)$ to an open subset containing $\Sigma$, there exists a bi-$H(F)$-invariant $R$-linear pairing 
		$$Z_\pi:\ \pi\times\tilde{\pi}\to R$$
		such that for any $x\in\Sigma$, the following
		diagram is commutative
		\[\xymatrix{
			&\pi \times \tilde{\pi} \ar[d] \ar[r]^-{Z_\pi} & R \ar[d] \\
			&\pi|_x \times \tilde{\pi}|_x \ar[r]^-{Z_{\pi|_x}} & k(x)
			.}\]
		Here both vertical arrows are the specialization maps. 
	\end{enumerate}
\end{conj}

It is natural to consider the rationality and meromorphy property of  spherical characters.  
\begin{conj}\label{conj-1} Let $\pi$  be a finitely generated smooth admissible  torsion-free $R[G(F)]$-module such 
	that  $\CE(\pi|_x) \neq\emptyset$ for any $x\in\Sigma$. Then 	\begin{enumerate}
		\item (Rationality) for any $x\in\Sigma$, there exists a unique character
		$$J_{\pi|_x}:\  \CS(G(F),k(x))\to k(x)$$ 
such that for any $\tau\in\CE(\pi|_x)$, the following diagram commutes
$$\xymatrix{ \CS(G(F),k(x)) \ar[d]^{\tau} \ar[r]^{\quad   J_{\pi|_x}} & k(x) \ar[d]^{\tau} \\
	\CS(G(F),\BC) \ar[r]^{\quad J_{(\pi|_x)_\tau}}  & \BC}$$	
		\item (Meromorphy) Moreover for any $f\in \CS(G(F),E)$, there exists a unique meromorphic function 
			$J_\pi(f)\in \Frac R$  interpolating $J_{\pi|_x}(f)$, $x\in\Sigma$.
	\end{enumerate}
\end{conj}

In general, we don't know whether the above two conjectures are equivalent. 
However, we can prove the following implication result.

\begin{prop}\label{Imp} When $R=E$, the rationality part in Conjecture \ref{conj-2} implies that of Conjecture 
	\ref{conj-1}. For general $R$,  the meromorphy part of Conjecture \ref{conj-2} 
implies that of Conjecture \ref{conj-1} for  those $\pi$ whose  fiber rank of $\pi^K$ is locally constant on $\Sigma$ for all open compact subgroups $K\subset G(F)$, i.e. the function 
     $$\phi_{\pi^K}:\ X\mapsto \BN;\quad x\mapsto\dim_{k(x)}(\pi^K|_x)$$
     is locally constant  on $\Sigma$. 
\end{prop}
\begin{proof}Take any open compact subgroup $K\subset G(F)$
and $f\in\CH(K,E)$. For the case $R=E$, one can simply take $$J_\pi(f)=\sum_i \frac{Z_\pi(\pi(f)\varphi_i, \varphi^i)}{(\varphi_i,\varphi^i)}$$
where  $\{\varphi_i\}$ (resp. $\{\varphi^i\}$) is any basis of $\pi^K$ (resp. $\pi^{\vee,K}$) such that  $(\varphi_i, \varphi^j)=0$ for any $i\neq j.$ 

For the general case, upon shrinking $X$ to an open neighborhood of $\Sigma$ if necessary and by gluing, we may
assume  $\pi^K, \wt{\pi}^K$ are free $R$-modules and the induced pairing   $(\cdot,\cdot):\ \pi^K\times\wt{\pi}^K\to R$ is perfect by  \cite[{Lemma 0FWG}]{SP} and the local constancy assumption.

Take  bases $\{\varphi_i\}$ (resp. $\{\wt{\varphi}_i\}$) of $\pi^K$
(resp. $\wt{\pi}^K$) such that $(\varphi_i,\wt{\varphi}_j)=\delta_{ij}$. Clearly
$$J_\pi(f):=\sum_i Z_\pi(\pi(f)\varphi_i,\wt{\varphi}_i)$$
  is meromorphic. Note that  for any $x\in\Sigma$, the natural map 
$$\Hom_{R}(\pi^K,R)\otimes k(x)\to \Hom_{k(x)}(\pi^K|_x,k(x))$$
is an isomorphism. By taking inductive limits, one finds the natural map $\pi^\vee|_x\to (\pi|_x)^\vee$ is an isomorphism for any $x\in\Sigma$. Consequently,  $J_\pi$ interpolates $J_{\pi|_x}(f)$, $x\in\Sigma$.
 \end{proof}
\begin{remark}Perhaps the local constancy assumption  holds for any smooth admissible finitely generated torsion-free $R[G(F)]$-module $\pi$ such that $\pi|_x$ is absolutely irreducible for any $x\in\Sigma$. For the case $G=\GL_n$, see Proposition \ref{local constancy} below.
\end{remark}

Now we consider spherical characters in the unitary Gan-Gross-Prasad case. Let $W_n$ be a Hermitian space of dimension $n \geq 2$ with respect to  a given  quadratic field extension $F^\prime/F$, and $w \in W_n$ be an anisotropic vector. Let $U_n$ be the unitary group associated to $W_n$ and 
$U_{n-1}$ the stabilizer of $w$ in $U_n$.  We will consider spherical characters for the strongly tempered spherical variety $Y:=H\backslash G$ where  $H = U_{n-1}$ embeds into $G = U_n \times U_{n-1}$ diagonally. 

In this case, Conjecture \ref{conj-2} holds by \cite[Proposition 4.13]{CF21}. Together with Proposition \ref{Imp}, one  has the following corollary.
\begin{cor}\label{GGP R}The rationality part of Conjecture \ref{conj-1} holds for $Y$.
\end{cor}
However, the local constancy of fiber ranks is not known in general, even in the  global setting for our motivating applications. To avoid this issue, we consider   the quadratic  base change functoriality  $\pi\to \BBC(\pi)$   from irreducible smooth admissible complex representations over $G(F)$ to those over $G^\prime(F)$  established in \cite[Theorem 3.6.1]{Mok15} and \cite[Theorem 1.6.1]{KMSW14} so that we can apply the local constancy for families of $G^\prime$-representations. Here  $G^\prime$ is the $F$-algebraic group $\Res_{F^\prime/F}(\GL_{n}\times\GL_{n-1})$. 

Fix an algebraic closure $\bar{E}$ of $E$ and a field embedding $\tau:\ \bar{E}\to \BC$. Our main result is the following:
\begin{thm}\label{GGP case} Assume that
\begin{enumerate}[(i)]
    \item  For any $x\in\Sigma$, there exists an absolutely irreducible smooth admissible  $G^\prime(F)$-representation $\BBC(\pi|_x)$  over $k(x)$ such that   $\BBC(\pi|_x)_\tau\cong \BBC((\pi|_x)_\tau)$;
    \item  There exists a smooth admissible finitely generated torsion-free $R[G^\prime(F)]$-module $\BBC(\pi)$ such that  $\BBC(\pi|_x)\cong\BBC(\pi)|_x$ for all $x\in\Sigma$.
\end{enumerate}
Then  Conjecture \ref{conj-2} holds for $\pi$. 
\end{thm}
In the interested global setting, one usually takes $E=\bar{\BQ}_p$ and  $\tau$ to be an isomorphism. Then Item (i) is automatic and  Item (ii)  can be deduced from the existence of family of Galois representations, 
local-global compatibility and the local Langlands correspondence in family.
 
Under the assumption on base change functoriality, we approach Theorem \ref{GGP case} as following:
\begin{enumerate}[(Step 1)]
	\item  Apply  the spherical character identity of  Beuzart-Plessis (See \cite{BP20, Iso})
\[J_\pi(f) \stackrel{.}{=} I_{\BBC(\pi)}(f')\]
for any $f\in \CS(G(F),\BC)$ and 
$f'\in \CS(G^\prime(F),\BC)$
with purely matching orbital integrals to transfer $J_\pi$ to the spherical character  $I_{\BBC(\pi)}$ for the Rankin-Selberg and Asai
zeta integrals. 
To apply the character identity,  the existence of smooth transfer for  $E$-valued
Schwartz functions is needed (See Proposition \ref{RST} below).
	\item Apply the theory of co-Whittaker modules to establish
the meromorphy property of Rankin-Selberg and Asai zeta integrals in families (initialed in \cite{Mos16G}, see Proposition \ref{Mero RS} and \ref{Mero As} below). 

To deduce the meromorphy of $I_{\BBC(\pi)}$, the local constancy of fiber rank in the $\GL_n$-case is needed (See Proposition \ref{local constancy} below).
\end{enumerate}

\s{\bf Acknowledgement} We express our sincere gratitude to Prof. Y. Tian for his consistent encouragement. We thank  Prof. A. Burungale for inspiring discussions.

L. Cai is partially supported by NSFC grant No.11971254.

\section{The unitary Gan-Gross-Prasad case}
In this section, we establish Theorem \ref{GGP case}. Throughout, 
\begin{itemize}
    \item let $\CO\subset F$ be the ring of integers and $q$ be the cardinality of the residue field of $\CO$;
    \item  fix an unramified nontrivial additive character $\psi:\ F\to\bar{E}^\times$ and set $E_\psi:=E(\psi(a)| a\in F)$;
    \item all measures on $F$-points of $F$-groups are Haar measures such that volumes of open compact subgroups are rational numbers. 
\end{itemize}  
To simplify notations, we assume $E$ contains a square root $\sqrt{q}$ of $q$. Otherwise, we can first work on $E(\sqrt{q})$ and then  descent by Corollary \ref{GGP R}.

  \subsection{Zeta integral in  families}
We start with the meromorphy property of Rankin-Selberg and Asai zeta integral in  families. 
	Set $G_n:=\GL_n(F)$ and let $N_n\subset P_n\subset G_n$ be the upper-triangular unipotent subgroup and the mirabolic subgroup respectively.    Then by \cite[Section 3.1]{EH14},  there are 
	\begin{itemize}
		\item the functor  $\pi\mapsto\CJ(\pi)$ from smooth $R[G_n]$-modules to $R[P_n]$-modules
		\item  the 
		functor $\pi\mapsto\pi^{(n)}$ from  smooth $R[G_n]$-modules to  $R$-modules
	\end{itemize}
	such that for any $R$-module $M$, $$\CJ(\pi\otimes_{R} M)=\CJ(\pi)\otimes_{R} M,\quad (\pi\otimes_{R}M)^{(n)}=\pi^{(n)}\otimes_{R}M.$$
	Extend the additive character $\psi$  to $N_n$ by $n\in N_n\mapsto \psi(\sum_{i=1}^{n-1}n_{i,i+1})$. Then after base change to $R\otimes_{E}E_\psi$ (see \cite[Proposition 4.1.2]{Dis20}), $$\pi^{(n)}=\pi/\pi(N_n,\psi),\quad \pi(N_n,\psi):=\langle n\cdot v-\psi(n)v\mid n\in N_n,\ v\in\pi\rangle.$$
\begin{defn}	A smooth admissible $R[G_n]$-module $\pi$ is   {\em of Whittaker type} if the $R$-module $\pi^{(n)}$ is  locally free  of rank one.  A $R[G_n]$-module $\pi$ of Whittaker type is called {\em co-Whittaker} if $\sigma^{(n)}\neq0$ for any non-zero $R[G_n]$-quotient $\sigma$ of $\pi$, or equivalently $\CJ(\pi)$ generates $\pi$ (see \cite[Lemma 4.2.1]{Dis20}).

 For any smooth admissible $R[G_n]$-module $\pi$ of Whittaker type, the {\em space of Whittaker functions} $\CW(\pi,\psi)$ of $\pi$ is the image of the  map 
$$(\pi^{(n)})^*\otimes_{R}\pi\otimes_{E}E_\psi\lra (\Ind_{N_n}^{G_n}E_\psi)\otimes_{E}R$$ 
induced by the canonical isomorphism of $R\otimes_EE_\psi$-modules
$$(\pi^{(n)}\otimes_{E}E_\psi)^*\cong \sheafhom_{G_n}(\pi\otimes_{E}E_\psi,\Ind_{N_n}^{G_n}E_\psi\otimes_{E}R).$$
Here in $\Ind_{N_n}^{G_n}E_\psi$, $N_n$ acts on $E_\psi$ via $\psi$.  Note that for any closed point $x\in X=\Spec(R)$, there is a natural surjection $\CW(\pi,\psi)|_x\to \CW(\pi|_x,\psi)$ which becomes an isomorphism when $\pi|_x$ is irreducible.
\end{defn} 
	For any  smooth $R[G_n]$-module $\pi$ and any standard parabolic $P=MN\subset G_n$ with  Levi factor $M$, the Jacquet module
	$J_M(\pi):= \pi/\langle n\cdot v-v\mid n\in N,\ v\in\pi\rangle$ is a smooth  $R[M]$-module.
\begin{prop}\label{finite}If $\pi$ is co-Whittaker,  $J_M(\pi)$ is finitely generated  and admissible. In particular, the $R$-module $\End_{R[M]}(J_M(\pi))$ is coherent.
	\end{prop}
	\begin{proof}By \cite[Lemma 2.29]{Mos16}, $\pi$ is  finitely generated and consequently, $J_M(\pi)$ is finitely generated. The admissibility of $J_M(\pi)$ is a special case of \cite[Corollary 1.5]{DHK22}, see also  \cite[Theorem 10.7 \& 10.9]{Hel16b}. By \cite[lemma 4.1.1]{Dis20}, the  coherence of $\End_{R[M]}(J_M(\pi))$ follows. 
	   	\end{proof}

Now we  consider the Rankin-Selberg  zeta integral in families.  Take positive integers $m\leq n$ and let $\pi_1$ and  $\pi_2$ be smooth admissible $R[G_n]$-module and $R[G_m]$-module of Whittaker type respectively. 
For  $W_1\in\CW(\pi_1,\psi)$, $W_2\in \CW(\pi_2,\psi^{-1})$ and $\Phi\in \CS(F^n,E)\otimes_ER$, consider the following formal  series
$$J_{RS}(W_1,W_2,T):=\sum_{j\in\BZ}J^j(W_1,W_2)T^j \quad Z_{RS}(W_1,W_2,(\Phi),T):=\sum_{j\in\BZ}Z_{RS}^j(W_1,W_2,(\Phi))T^j$$
in the variable $T$ when $m=n$ and $m=n-1$ where
\begin{itemize}
	\item for $m=n$,   $$J_{RS}^j(W_1,W_2):=\int_{N_{n-1}(F)\backslash \GL_{n-1}^j(F)}W_1(g)W_2(g) dg,$$  $$Z_{RS}^j(W_1,W_2,\Phi):=\int_{N_n(F)\backslash \GL_{n}^j(F)}W_1(g)W_2(g)\Phi(e_ng) dg $$
	with $F^n$ viewed as row vectors and $e_n=(0,\cdots,0,1)$,
	\item for $m=n-1$, $$ Z_{RS}^j(W_1,W_2):=\int_{N_{m}(F)\backslash \GL_{m}^j(F)}W_1(\begin{pmatrix}g & 0\\ 0 & 1\end{pmatrix})W_2(g)dg.$$
\end{itemize}  
Here for any subgroup $H\subset G_n$ and any $j\in\BZ$, $H^j:=\{g\in H\mid \val_F(\det(g))=j\}$.   Note that by the Iwasawa decomposition,  the integrals $Z_{RS}^j(W_1,W_2,(\Phi))$ and $J_{RS}^j(W_1,W_2)$) are  actually  finite sum for each $j$, and hence  $Z_{RS}(W_1,W_2,(\Phi),T)$ and $J_{RS}(W_1,W_2,T)$)  are well-defined. Moreover when $X=\Spec(\BC)$ and $\Re(s)\gg0$,  the substitution $T=q^{-(s-\frac{n-m}{2})}$ in  $Z_{RS}(W_1,W_2,(\Phi),T)$ gives the usual Rankin-Selberg Zeta integral.

\begin{prop}\label{Mero RS} Let $S\subset R[T]$ be the  multiplicative subset consisting of polynomials whose leading and trailing coefficients are units. Let $\pi_1$ and $\pi_2$ be   co-Whittaker $R[G_n]$ and $R[G_m]$)-modules respectively.   Then for any $W_1\in\CW(\pi_1,\psi)$, $W_2\in\CW(\pi_2,\psi^{-1})$ and $\Phi\in \CS(F^n,E)\otimes_ER$, $Z_{RS}(W_1,W_2,(\Phi),T)$ and  $J(W_1,W_2,T)$  belong to $S^{-1}(R[T,T^{-1}])\otimes_EE_\psi$.
\end{prop}
\begin{proof} The case $Z_{RS}(W_1,W_2,T)$ with $m=n-1$ is treated in \cite[Theorem 3.2]{Mos16G}. We now deal with  $Z_{RS}(W_1,W_2,\Phi,T)$ and  $J(W_1,W_2,T)$ when $n=m$.
	
Let $A_n\subset G_n$ be the subgroup of diagonal matrices and equip $A_n$ with the coordinates
$$(F^\times)^n\xrightarrow{\sim} A_n;\quad a=(a_1,\cdots,a_n)\mapsto \diag\{a_1\cdots a_n,a_2\cdots a_n,\cdots,a_n\}.$$
Then (see  \cite[Lemma 3.2]{Mos16})  for any $W\in\CW(\pi,\psi)$, there exists a constant $C$  such that $W(a)=0$ unless $\val_F(a_i)>-C$ for $i=1,\cdots,n-1$.
Then  $Z(W_1,W_2,\Phi,T)$ and $J(W_1,W_2,T)$ are  actually formal Laurent series. Moreover by the Iwasawa decomposition, it suffices to show the formal  Laurent series 
	\begin{align*}Z^\prime(W_1,W_2,\Phi,T):&=\sum_{j\in\BZ}\big(\int_{A_n^j(F)}W_1(a)W_2(a)\Phi(e_n a)da\big)T^j,\\
	J^\prime(W_1,W_2,T):&=\sum_{j\in\BZ}\big(\int_{A_{n-1}^j(F)}W_1(\diag\{a^\prime,1\})W_2(\diag\{a^\prime,1\})da^\prime\big) T^j
	\end{align*}
	belong to  $S^{-1}(R[T,T^{-1}])\otimes_EE_\psi$. 

By \cite[Proposition 6.2]{Hel16a}, $\pi_i$ admits a central character $\chi_i$ for $i=1,2$. 	Let  
	$$Z(\chi,\Phi_n,T):=\sum_{j\in\BZ}\big(\int_{\varpi^j\CO_F^\times}\chi(x)\Phi_n(x)dx\big)T^{j}$$
	where $\Phi_n(x):=\Phi(0,\cdots,0,x)\in\CS(F,E)\otimes_ER$ and $\chi:=\chi_1\chi_2$.
	Then in $R[[T]][T^{-1}]\otimes_EE_\psi$, 
	$$Z^\prime(W_1,W_2,\Phi,T)=J^\prime(W_1,W_2,T)Z(\chi,\Phi_n,T^n).$$
	Straightforward computation shows $$Z(\chi,\Phi_n,T^n)\in S^{-1}(R[T,T^{-1}])\otimes_EE_\psi.$$ Hence we only need to show   $$J^\prime(W_1,W_2,T)\in S^{-1}(R[T,T^{-1}])\otimes_EE_\psi.$$

	Consider the coordinates 
	$$(F^\times)^{n-1}\xrightarrow{\sim} A_{n-1};\quad a^\prime=(a_1,\cdots,a_{n-1})\mapsto \diag\{a_1\cdots a_{n-1},a_2\cdots a_{n-1},\cdots,a_{n-1}\}.$$
	Let $\CV\subset C^\infty(A_{n-1},R\otimes_E E_\psi)$ be the subspace generated by functions of the form
	$$a^\prime\mapsto W_1(\diag\{a^\prime,1\})W_2(\diag\{a^\prime,1\}), W_1\in\CW(\pi_1,\psi),\ W_2\in\CW(\pi_2,\psi^{-1})$$
 and $\CV_i\subset\CV$ be the subspace of functions $\phi$ satisfying that $\phi(a)=0$ when $\val_F(a_i)\geq C$ for some constant $C$ (depending on $\phi$) for each $i\leq{n-1}$. Moreover, let $M_i\subset G_n$ be the standard Levi subgroup $G_{i}\times G_{n-i}$ for $1\leq i\leq {n-1}$. Then similar as \cite[Lemma 3.3-3.5]{Mos16G}, the map 
	$$\CW(\pi_1,\psi)\otimes 	\CW(\pi_2,\psi^{-1})\to \CV\to\CV/\CV_i$$
	factors through $J_{M_i}(\CW(\pi_1,\psi))\otimes J_{M_i}(\CW(\pi_2,\psi^{-1}))$. 	Let 
	$\rho_i(\varpi)$ be the right translation by $\diag\{\varpi,\cdots,\varpi,1,\cdot,1\}$ (the first $i$ entities are $\varpi$) acting diagonally on $\pi_1\otimes\pi_2$. Then $\rho_i(\varpi)$ induce an element in  $\End_{R[M_i\times M_i]}(J_{M_i}(\pi_1)\otimes J_{M_i}(\pi_2))$, which is coherent by Proposition \ref{finite}. Then  as \cite[Lemma 3.6]{Mos16G}, there exist polynomials $f_i\in S$, $1\leq i\leq n-1$ such that 
	$$\big(\prod_{i=1}^{n-1}f_i(\rho_i(\varpi))\big)(W_1\otimes W_2)\mid_{A_{n-1}}\in\cap_{i=1}^{n-1} \CV_i.$$
	Consequently, we have $$J^\prime\big((\prod_{i=1}^{n-1}f_i(\rho_i(\varpi)))W_1, (\prod_{i=1}^{n-1}f_i(\rho_i(\varpi)))W_2,T\big)\in R[T,T^{-1}]\otimes_EE_\psi.$$
	Note that 
	$$T^i J^\prime(\rho_i(\varpi)W_1, \rho_i(\varpi)W_2,T)=J^\prime(W_1, W_2,T),$$ thus one deduces
	$J^\prime(W_1,W_2,T)\in S^{-1}(R[T,T^{-1}])\otimes_EE_\psi$. We are done.
\end{proof}

Take $x\in \Sigma$ and assume there exists a field embedding  $\tau:\ k(x)\hookrightarrow\BC$ such that $(\pi|_x)_\tau$ is essentially unitary and irreducible.
Then by \cite{JS81}, for any ring morphism $\tau:\ k(x)\otimes_E E_\psi\to \BC$ extending $\tau$,  the pairing given by  absolutely convergent integration	$$\CW(\pi|_x,\psi)\otimes_{k(x)\otimes_E E_\psi,\tau}\BC\times\CW(\pi|_x,\psi^{-1})\otimes_{k(x)\otimes_E E_\psi,\tau}\BC\to\BC$$
	$$ (W,\wt{W})\mapsto \int_{N_{n-1}\backslash G_{n-1}}W(g)\wt{W}(g)dg= \sum_{j\in\BZ}J_{RS}^j(W,\wt{W})$$
 	is  $G_n$-equivariant and non-degenerate.
Consequently, 
\begin{itemize}
	\item $T=1$ is not a pole of $J_{RS}(W_x,\wt{W}_x,T)$ for any $W_x\in\CW(\pi|_x,\psi)$ and $\wt{W}_x\in\CW(\wt{\pi}|_x,\psi^{-1})$;
	\item the pairing 
	$$\langle-,-\rangle_{x}:\ \CW(\pi|_x,\psi)\times \CW(\wt{\pi}|_x,\psi^{-1})\to k(x)\otimes_EE_\psi;\quad (W_x,\wt{W}_x)\mapsto J_{RS}(W_x,\wt{W}_x,1)$$
	is $G_n$-equivariant and non-degenerate.
\end{itemize}
From these observations, we immediately deduce the following corollary of Proposition \ref{Mero RS}.
\begin{cor}\label{InvWM}Let $\pi$ be a  co-Whittaker $R[G_n]$-module such that for any $x\in\Sigma$, $(\pi|_x)_\tau$ is essentially unitary and irreducible for some field
embedding $\tau:\ k(x)\hookrightarrow\BC$. Then there exists an open subset $X^\prime\subset \Spec(R)$ containing $\Sigma$ such that for any $W\in\CW(\pi,\psi)$ and $\wt{W}\in \CW(\wt{\pi},\psi^{-1})$, $J_{RS}(W,\tilde{W},1)$ is regular on $X^\prime\times_E E_\psi$. 
	Moreover,  the  pairing 	
	$$\langle-,-\rangle:\ \CW(\pi|_{X^\prime},\psi)\times\CW(\tilde{\pi}|_{X^\prime},\psi^{-1})\to\CO_{X^\prime}\otimes_EE_\psi;\quad (W,\wt{W})\mapsto J_{RS}(W,\wt{W},1)$$
	is $G_n(F)$-invariant and	interpolates $\langle-,-\rangle_{x}$ for any $x\in\Sigma$.
\end{cor}

Now we consider Asai zeta inetgrals in families. Let $G_n^\prime=\GL_n(F^\prime)$ and  $\pi$ be a finitely generated smooth admissible $R[F^\times\bs G_n^\prime]$-module of Whittaker type. Let $\psi_{F^\prime}$ be the additive character of $F^\prime$ given by $a\mapsto\psi(\frac{1}{2}\RTr_{F^\prime/F}(a))$.
Fix $\xi\neq0\in F^\prime$ such that $\RTr_{F^\prime/F}(\xi)=0$ and set $\epsilon_n(\xi)=\diag\{\xi^{n-1},\cdots,\xi,1\}\in G_n^\prime$.
 For any Whittaker function $W\in\CW(\pi,\psi_{F^\prime})$ and any Schwartz function  $\Phi\in \CS(F^n,E)\otimes R$, consider the formal power series
\begin{align*}
Z_{As}(W,\Phi,T)&:=\sum_{j\in\BZ}Z_{As}^j(W,\Phi)T^j,\quad  Z_{As}^j(W,\Phi):=\int_{N_{n}\backslash G_n^j}W(\epsilon_n(\xi)g)\Phi(e_ng)\eta(\det(g))^{n-1}dg;\\
J_{As}(W,T)&:=\sum_{j\in\BZ}J_{As}^j(W)T^j,\quad  J_{As}^j(W):=\int_{N_{n-1}\backslash G_{n-1}^j}W(\epsilon_n(\xi)g)\eta(\det(g))^{n-1}dg.
\end{align*}
Here    $\eta$ is the  quadratic character associated to the quadratic field extension $F^\prime/F$. 
As in the Rankin-Selberg  case, the integration $Z_{As}^j(W,\Phi)$ and $J_{As}^j(W)$ are actually finite sum for each $j\in\BZ$, and  
$Z_{As}(W,\Phi,T)$ and $J_{As}(W,T)$ are formal Laurent series. Moreover, the substitution $T=q^{-s}$ in  $Z_{As}(W,\Phi,T)$ gives the usual Asai zeta integral when $X=\Spec(\BC)$ and $\Re(s)\gg0$. 

We have the following analogue of Proposition \ref{Mero RS}.
\begin{prop}\label{Mero As} Let $\pi$ be a  co-Whittaker $R[G_n^\prime]$-module. Then for any $W\in\CW(\pi,\psi_{F^\prime})$ and $\Phi\in \CS(F^n,E)\otimes_E R$,  $Z_{As}(W,\Phi,T)$  and $J_{As}(W,T)$ belong to $S^{-1}(R[T,T^{-1}])\otimes_EE_\psi$.
\end{prop}
Let $\pi$ be a  co-Whittaker $R[F^\times\backslash G_n^\prime]$-module. Take $x\in \Sigma$ and assume there exists a field embedding $\tau:\ k(x)\hookrightarrow\BC$ such that $(\pi|_x)_\tau$ is essentially unitary and irreducible. Then  
 by \cite[Page 185]{GJR01} and \cite[Section 3.2]{WeZ14}, for any  morphism $\tau:\ k(x)\otimes_E E_\psi\rightarrow\BC$ extending $\tau$, the linear functional given by absolutely convergent integration 
	$$\CW(\pi|_x,\psi_{F^\prime})\otimes_{k(x)\otimes_E E_\psi,\tau}\BC\to\BC$$
	$$ W\mapsto \int_{N_{n-1}\backslash G_{n-1}}W(\epsilon_n(\xi)g)\eta(\det(g))^{n-1}dg=\sum_{j\in\BZ}J_{As}^j(W)$$
	is  non-zero and $\eta^{n-1}$-equivariant (with respect to the $G_n$-action). Consequently, 
\begin{itemize}
	\item $T=1$ is not a pole of $J_{As}(W_x,T)$ for any $W_x\in\CW(\pi|_x,\psi_{F^\prime})$
	\item the linear functional
	$$\ell_{x}:\ \CW(\pi|_x,\psi_{F^\prime})\to k(x)\otimes_EE_\psi;\quad W_x\mapsto J_{As}(W_x,1)$$
is $\eta^{n-1}$-equivariant (with respect to the $G_n$-action).
\end{itemize}
Similar to the Rankin-Selberg case, we have the following corollary of Proposition \ref{Mero As}:  
\begin{cor}\label{Beta As} Let $\pi$ be a  co-Whittaker $R[F^\times\backslash G_n^\prime]$-module such that for any $x\in \Sigma$, there  exists a field embedding $\tau:\ k(x)\hookrightarrow\BC$ such that $(\pi|_x)_\tau$ is essentially unitary and irreducible. Then there exists an open subset $X^\prime\subset \Spec(R)$ containing $\Sigma$ such that	$J_{As}(W,1)$ is regular on $X^\prime\times_EE_\psi$  for any $W\in\CW(\pi,\psi_{F^\prime})$.
	Moreover, the  linear functional 	
	$$\ell:\ \CW(\pi|_{X^\prime},\psi_{F^\prime})\to\CO_{X^\prime}\otimes_E E_\psi;\quad W\mapsto J_{As}(W,1)$$
	is $\eta^{n-1}$-equivariant and interpolates $\ell_{x}$ for all $x\in\Sigma$.
\end{cor}

 \subsection{Rationality  of smooth matching} Consider the $F$-subgroups  $H_1^\prime:=\Res_{F^\prime/F}\GL_{n-1}$ (embedded diagonally) and $H_2^\prime:=\GL_{n}\times \GL_{n-1}$ of $G^\prime$. According to \cite[Section 3.1]{Zha14}, there is a canonical isomorphism between categorical quotients
$$H\backslash G/H\cong H_1^\prime\backslash G^\prime/H_2^\prime.$$
A point  $\delta\in G$ (resp. $\gamma\in  G^\prime$) is called {\em regular semisimple} with respect to the action of $H\times H$ (resp. $H_1^\prime\times H_2^\prime$) if its orbit is closed and its stabilizer is trivial. Let $G_{rs}\subset G$ (resp. $G_{rs}^\prime\subset G^\prime$) be the open subset consisting of regular semisimple elements. Then we have an induced injection 
\begin{equation}
\label{orbit} H(F)\backslash G_{rs}(F)/H(F)\hookrightarrow H_1^\prime(F)\backslash G_{rs}^\prime(F)/H_2^\prime(F).
\end{equation}
The orbital integrals associated to $\delta\in G_{rs}(F)$ and $\gamma\in G^\prime_{rs}(F)$ are the following distributions respectively 
$$f\in \CS(G(F),\BC)\mapsto O(\delta,f):=\int_{(H\times H)(F)}f(h\delta h^\prime)d\mu_H(h) d\mu_H(h^\prime),$$
$$f^\prime\in \CS(G^\prime(F),\BC)\mapsto O(\gamma,f^\prime):=\int_{(H_1^\prime\times H_2^\prime)(F)}f^\prime(h_1^{-1}\gamma h_2)\eta(\det(h_2))d\mu_{H_1^\prime}(h_1) d\mu_{H_2^\prime}(h_2)$$
where $\eta$ is the character $$\eta:\ H_2^\prime(F)=\GL_n(F)\times \GL_{n-1}(F)\to\{\pm1\};\quad (h_1,h_2)\mapsto \eta(\det(h_2))^n\eta(\det(h_1))^{n-1}.$$  As $\delta$ and  $\gamma$ are regular semisimple,  the orbital integrals are actually finite sums and $\Aut(\BC)$-equivariant in the sense that for any $f\in \CS(G(F),\BC)$, $f^\prime\in \CS(G^\prime(F),\BC)$ and $\sigma\in\Aut(\BC)$,
$$O(\gamma,f)^\sigma=O(\gamma,\sigma\circ f),\quad  O(\gamma,f^\prime)^\sigma=O(\gamma,\sigma\circ f^\prime).$$

Attached to the unique quadratic character on $F^{\prime,\times}$ which extends $\eta$, there exists  (see \cite[Section 3.4]{BP20})  a {\em transfer factor} $$\Omega: G_{rs}^\prime(F)\to\{\pm1\},\quad \Omega(r\cdot h)=\eta(h)\Omega(r)\quad \forall\ h\in H_2^\prime(F).$$
A {\em pure smooth transfer} of $f\in \CS(G(F),E)$ is a Schwartz function $f^\prime\in \CS(G^\prime(F),E)$ such that 
\begin{itemize}
	\item For every $\delta\in G_{rs}(F)$ and $\gamma\in G^\prime_{rs}(F)$ whose orbits corresponding to each other via the injection \ref{orbit}, one has 
	$O(\delta,f)=\Omega(\gamma)O(\gamma,f^\prime).$
	\item For every $\gamma\in G^\prime_{rs}(F)$ whose orbit lies outside the image of the injection \ref{orbit}, one has
	$O(\gamma,f^\prime)=0.$
\end{itemize}
\begin{prop}\label{RST}The pure smooth transfer exists  for any  $f\in \CS(G(F),E)$. 
\end{prop}
\begin{proof}Fix an embedding $E\subset\BC$. For any $f\in \CS(G(F),E)\subset \CS(G(F),\BC)$,  there exists a pure smooth transfer $f^\prime\in \CS(G^\prime(F),\BC)$ by \cite[Theorem 2.6]{Zha14}.   Take any $E$-linear splitting  $p_E:\ \BC\to E$ of the inclusion $E\subset \BC$ and set $f_E^\prime:=p_E\circ f^\prime\in \CS(G(F),E)$.	
	Let $N:=\ker p_E$ and $f_N^\prime:=f^\prime-f_E^\prime$. Then for any $\gamma\in G_{rs}^\prime(F)$, we have $$\Omega(\gamma)O(\gamma,f_N^\prime)=(1-p_E)\Omega(\gamma)O(\gamma,f^\prime)\in N.$$
	If the orbit of $\gamma$ matches with that of  $\delta\in G_{rs}(F)$, then $\Omega(\gamma)O(\gamma,f_N^\prime)\in E$. In fact the $\Aut(\BC)$-equivalence of orbital integral implies that for any $\sigma\in\Aut(\BC/E)$ (note that $\Omega(\gamma)\in\{\pm1\}$)
	\begin{align*}
	\Omega(\gamma)O(\gamma,f_N^\prime)^\sigma&=O(\delta, f)^\sigma-\Omega(\gamma)O(\gamma,f_E^\prime)^\sigma\\
	&=O(\delta, f)-\Omega(\gamma)O(\gamma,f_E^\prime)=\Omega(\gamma)O(\gamma,f_N^\prime)
	\end{align*}  Combining these, we find $\Omega(\gamma)O(\gamma,f_N^\prime)=0$ and $f_E^\prime$ is a pure smooth transfer of $f$.
\end{proof}
\subsection{GGP spherical character in families}
Let $\sigma=\sigma_n\otimes_E \sigma_{n-1}$ be an irreducible smooth admissible $E$-representation of $(F^\times\times F^\times)\bs G^\prime(F)$ such that  $\sigma_\tau$ is tempered for some field embedding $\tau: E\hookrightarrow\BC$.  Set $$\CW(\sigma,\psi_{F^\prime}):=\CW(\sigma_n,\psi_{F^\prime})\otimes_{E_\psi}\CW(\sigma_{n-1},\psi_{F^\prime}^{-1}).$$
By	Proposition \ref{Mero RS} and \ref{Mero As} for $R=E$, there exist $E_\psi$-linear functionals 
	\begin{align*}
	\ell_{RS,\sigma}:\ & \CW(\sigma,\psi_{F^\prime})\to  E_\psi,\quad  W_n\otimes W_{n-1}\mapsto Z_{RS}(W_n,W_{n-1},1);\\
	\ell_{As,\sigma^\vee}:\ &\CW(\sigma^\vee,\psi_{F^\prime}^{-1})\to  E_\psi,\quad  W_n^\vee\otimes W_{n-1}^\vee\mapsto J_{As}(W_n^\vee,1)J_{As}(W_{n-1}^\vee,1).
	\end{align*}
	Moreover by Corollary \ref{InvWM}, there is a $G^\prime(F)$-invariant non-degenerate pairing  
$$(-,-):\ \CW(\sigma,\psi_{F^\prime})\times\CW(\sigma^\vee,\psi_{F^\prime}^{-1})\to E_\psi$$
$$	(W_n\otimes W_{n-1}, W_n^\vee\otimes W_{n-1}^\vee)\mapsto  J_{RS}(W_n,W^\vee_n,1)J_{RS}(W_{n-1},W_{n-1}^\vee,1).$$
\begin{defn}
Fix  isomorphisms of $G^\prime(F)$-representations	
$$\rho:\ \sigma\otimes_EE_\psi\cong\CW(\sigma,\psi_{F^\prime});\quad \rho^\vee:\ \sigma^\vee\otimes_EE_\psi\cong \CW(\sigma,\psi_{F^\prime}).$$	
Let $I_\sigma$ be the character 
$$\CH(G^\prime(F),E)\to E_\psi;\quad f^\prime\in \CH(K,E)\mapsto \epsilon(1/2,\eta,\psi)^{-\frac{n(n-1)}{2}}\sum_i\frac{\ell_{RS,\sigma}(\sigma(f^\prime)\rho(v_i))\ell_{As,\sigma^\vee}(\rho^\vee(v^i))}{(\rho(v_i),\rho^\vee(v^i))}.$$
Here $\epsilon(1/2,\eta,\psi)\in E_\psi^\times$ is the usual epsilon factor given by Gauss sum and   $\{v_i\}$,  $\{v^i\}$ are $E$-bases of $\sigma^K$ and $\sigma^{\vee,K}$ respectively such that $\langle v_i,v^j\rangle=\delta_{ij}$ for any non-degenerate $G^\prime(F)$-invariant $E$-linear pairing 
$\langle-,-\rangle:\ \sigma\times\sigma^\vee\to E.$ Note that  $I_\sigma$ is independent of $\langle-,-\rangle$.
\end{defn}

To consider $I_\sigma$ in families, we need the local constancy of fiber rank in the $\GL_n$-case. 
\begin{prop}\label{local constancy}Let  $\pi$ be a smooth admissible finitely generated torsion-free $R[G_n]$-module such that for any $x\in\Sigma$, $\pi|_x$ is absolutely irreducible and generic. Then for any open compact subgroup $K\subset G_n$, the function $$\phi_{\pi^K}:\ X\to\BN,\quad x\mapsto \dim_{k(x)}(\pi^K|_x)$$ is locally  constant on $\Sigma$.
\end{prop}
\begin{proof}We use freely the notations in \cite[Sections 3-5]{Dis20}. By \cite[Lemma 4.2.3]{Dis20}, shrinking $X$ to an open subset containing $\Sigma$ if necessary, we may assume $\pi$ is co-Whittaker. Since $\pi|_x$ is absolutely irreducible for $x\in\Sigma$, one can deduce that $\pi|_\eta$ is absolutely irreducible and generic for each generic point $\eta\in X$. It suffices to deal with the case $X$ is connected. Then by \cite[Theorem 2.2]{Hel16a} and \cite[Section 3.3]{Dis20}, there exists the classifying map
 $\alpha:\ X\to\fX_{n,\bar{E}}$ from $X$ to the Bernstein variety factors through the component $\fX_{[\fs],[t]}$ determined by the supercuspidal support of $\pi|_\eta$ in the extended Bernstein variety.  As explained in the proof of \cite[Lemma 5.2.2]{Dis20}, there exists an explicitly constructed   torsion-free co-Whittaker module $\fM$ over $\fX_{[\fs],[t]}$
  such that $\pi|_\eta\cong\alpha^*\fM|_\eta$ and the fiber rank of $\fM$ is locally constant.  Since 
  $\pi$ and $\alpha^*\fM$ are both torsion-free  co-Whittaker $R[G_n]$-modules, one deduces $\pi\cong\alpha^*\fM$ from  \cite[Lemma 4.2.7]{Dis20}. Consequently, the fiber rank of $\pi$ is locally constant.
  \end{proof}
Now we can prove Theorem \ref{GGP case}.
\begin{proof}[Proof of Theorem \ref{GGP case}] By \cite[Lemma 4.2.6]{Dis20}, shrinking shrinking $X$ to an open subset containing $\Sigma$ if necessary,
 there exist torsion-free co-Whittaker $R[G_n^\prime]$-module $\sigma_n$ and  $R[G_{n-1}^\prime]$-module $\sigma_{n-1}$ such that  $\BBC(\pi)\cong\sigma_n\otimes\sigma_{n-1}$ as $R[G^\prime(F)]$-modules. Let $\wt{\sigma}_n$ be the $R$-module  $\sigma_n$ equipped with twisted $G_n$-action 
	$\wt{\sigma}_n(g)v:=\sigma_n({}^tg^{-1})v.$ 	Then  for any point $x\in \Sigma$, $\wt{\sigma}_n|_x\cong (\sigma_n|_x)^\vee$. Similar construction and result apply to $\sigma_{n-1}$. 
Then by Proposition \ref{Mero RS} and \ref{Mero As}, there are  $R\otimes_EE_\psi$-linear functional
	$\ell_{RS}$ on $\sigma_n\otimes\sigma_{n-1}\otimes_E E_\psi$ and  $\ell_{As}$ on  $\wt{\sigma}_n\otimes\wt{\sigma}_{n-1} \otimes_E E_\psi$  interpolating  $\ell_{RS,\BBC(\pi|_x)}$ and  $\ell_{As,\BBC(\pi|_x)}$ for all $x\in\Sigma$ respectively.  
	
	By Proposition \ref{Imp} and Proposition \ref{local constancy}, there exists a unique  character  $$I_{\sigma_n\otimes\sigma_{n-1}}:\ \CS(G^\prime(F), E)\to  \Frac{R}\otimes _E E_\psi$$ interpolating $I_{\BBC(\pi|_x)}$, $x\in\Sigma$. Consider the character $$J_\pi:\ \CS(G(F),E)\to \Frac(R)\otimes_E E_\psi,\quad f\mapsto J_\pi(f):=I_{\sigma_n\otimes\sigma_{n-1}}(f^\prime)$$
where $f^\prime\in \CS(G^\prime(F),E)$ is any smooth transfer of $f$ provided by Proposition \ref{RST}. 
By the character identity in \cite[Theorem 3.5.7]{BP20}, $J_\pi$ is well-defined and interpolates $J_{\pi|_x}(f)$. By Corollary \ref{GGP R}, one has $J_\pi(f)\in \Frac R$ and  we are done. 
\end{proof}
\begin{remark}Actually by the character identity, one can show $I_\sigma$ takes value in $E$ for any irreducible smooth admissible $E$-representation of $(F^\times\times F^\times)\bs G^\prime(F)$ with non-empty $\CE(\sigma)$.
\end{remark}


\begin{thebibliography}{XXXX}
	\bibitem{BP20}R. Beuzart-Plessis, \emph{Comparison of local relative characters and the Ichino-Ikeda conjecture for unitary groups}. Journal of the Institute of Mathematics of Jussieu, 2020 1-52. 
\bibitem{Iso} R. Beuzart-Plessis, Y. Liu, W. Zhang and X. Zhu, 
	\emph{Isolation of cuspidal spectrum, with application to the Gan-Gross-Prasad conjecture}.
 Ann. of Math. (2) 194 (2021), no. 2, 519–584.
\bibitem{CF21} L. Cai and Y. Fan, \emph{Families of Canonical Local Periods on Spherical Varieties},  arXiv:2107.05921v2.
\bibitem{DHK22} J.-F. Dat, D. Helm, R. Kurinczuk and G. Moss, \emph{Finiteness for Hecke algebras of $p$-adic groups}, Arxiv eprint: 2203.04929
\bibitem{Dis20} D. Disegni, \emph{ Local Langlands correspondence, local factors, and zeta integrals in analytic families}. J. London Math. Soc. (2020), 101: 735-764.
\bibitem{EH14}M. Emerton and D. Helm, \emph{The local Langlands correspondence for  $\GL_n$  in families}. (English, French summary)
Ann. Sci. \'Ec. Norm. Sup\'er. (4) 47 (2014), no. 4, 655-722.
\bibitem{Fli91}  Y.   Flicker, \emph{On distinguished representations},  J. Reine Angew. Math. 418 (1991), 139–172.
\bibitem{GJR01}S. Gelbart, H. Jacquet and J. Rogawski, \emph{Generic representations for the unitary group in three variables}. (English summary)
\bibitem{Hel16a}D. Helm, \emph{Whittaker models and the integral Bernstein center for $\GL_n$}. (English summary)
Duke Math. J. 165 (2016), no. 9, 1597-1628.
\bibitem{Hel16b}D. Helm, \emph{
The Bernstein center of the category of smooth $W(k)[\GL_n(F)]$-modules}. (English summary)
Forum Math. Sigma 4 (2016), Paper No. e11, 98 pp.
\bibitem{JS81}H. Jacquet and J A. Shalika,\emph{On Euler products and the classification of automorphic representations {I}}. American Journal of Mathematics, 1981, 103(3): 499-558.

\bibitem{KMSW14}T. Kaletha, A. Minguez,  SW. Shin et al. \emph{Endoscopic classification of representations: inner forms of unitary groups}. arXiv preprint arXiv:1409.3731, 2014.




\bibitem{Mok15}C.P. Mok, \emph{Endoscopic classification of representations of quasi-split unitary groups}. (English summary)
Mem. Amer. Math. Soc. 235 (2015), no. 1108, vi+248 pp. ISBN: 978-1-4704-1041-4; 978-1-4704-2226-4
\bibitem{Mos16}G. Moss, \emph{Interpolating local constants in families}, Math. Res. Lett. 23 (2016) 1789-1817.
\bibitem {Mos16G} G, Moss, \emph{Gamma factors of pairs and a local converse theorem in families}, Int. Math. Res. Not. IMRN 2016, no. 16, 4903-4936.
\bibitem{ST14}SW, Shin and N. Templier, \emph{On fields of rationality for automorphic representations}. Compositio Mathematica, 2014, 150(12): 2003-2053.
\bibitem{SV} Y. Sakellaridis, A. Venkatesh, {\em Periods and harmonic analysis on spherical varieties}, 
 Asterisque, (396):360, 2017. 
\bibitem{SP} Authors, \emph{Stack projects}, eprint \url{https://stacks.math.columbia.edu}, 2018
\bibitem{WeZ14} W. Zhang,  \emph{Automorphic period and the central value of Rankin-Selberg $L$-function}, J. Amer. Math. Soc. 27 (2014), no. 2, 541-612.
\bibitem{Zha14} W. Zhang, \emph{Fourier transform and the global Gan-Gross-Prasad conjecture for unitary groups}. (English summary)
Ann. of Math. (2) 180 (2014), no. 3, 971-1049.
\end{thebibliography}
\end{document}